\documentclass[10pt]{amsart}

\usepackage{amsmath, amsfonts, amsthm, amssymb, graphicx, fullpage, enumerate}

\newtheorem{theorem}{Theorem}      
\newtheorem{lemma}{Lemma}

\newtheorem*{main theorem}{Main Theorem}   
\newtheorem*{thmA}{Theorem A}  
\newtheorem*{thmB}{Theorem B}    
\newtheorem*{thmC}{Theorem C}   
\newtheorem*{thmD}{Theorem D}

\theoremstyle{remark}   
 
\theoremstyle{definition}  
\newtheorem{definition}{Definition}  
     
\def\N{\mathbb{N}}     
         
\def\R{\mathbb{R}}     
\def\Z{\mathbb{Z}}      
\def\C{\mathbb{C}}  
\def\P{\mathcal{P}}

\def\norm#1{\|#1\|}

\def\implies{\Longrightarrow}

\begin{document}
\title{Polynomials and Primes in Generalized Arithmetic Progressions (Revised Version)}
\author{Ernie Croot\quad\quad\quad Neil Lyall \quad\quad\quad Alex Rice }
 
\address{School of Mathematics, Georgia Institute of Technology, Atlanta GA 30332, USA}
\email{ecroot@math.gatech.edu}
\address{Department of Mathematics, The University of Georgia, Athens, GA 30602, USA}
\email{lyall@math.uga.edu}
\address{Department of Mathematics, University of Rochester, Rochester, NY 14612, USA}
\email{alex.rice@rochester.edu} 
\subjclass[2000]{11B30}
\begin{abstract} We provide upper bounds on the density of a  symmetric generalized arithmetic progression lacking nonzero elements of the form $h(n)$ with $n\in \N$ or  $h(p)$ with $p$ prime for appropriate $h\in \Z[x]$. The prime variant can be interpreted as a multi-dimensional, polynomial extension of Linnik's Theorem. This version is a revision of the published version. Most notably, the properness hypotheses have been removed from Theorems \ref{main1} and \ref{main2}, and the numerology in Theorem \ref{main1} has been improved. 
\end{abstract}
\maketitle

\setlength{\parskip}{7pt} 
\section{Introduction}
Given $N\in \N$, it is rather trivial to determine, up to a multiplicative constant, how large an arithmetic progression of the form $A=\{xd : 1\leq x \leq L\} \subseteq [1,N]:=\{1,\dots,N\}$ can be before it is guaranteed to contain a nonzero perfect square. Specifically, $d^2\in A$ if $L\geq d$, so if $A$ contains no squares, then $L^2< Ld \leq N$, and therefore $L< \sqrt{N}$. To observe the sharpness of this bound, fix any prime $p\in [\sqrt{N}/2,\sqrt{N}]$ and let $A=\{xp: 1\leq x \leq p-1\}$. However, with an additional degree of freedom, say if $$A=\{x_1d_1+x_2d_2 : \ |x_1|\leq L_1, \ |x_2|\leq L_2 \} \subseteq [-N,N],$$ the analogous question becomes far less trivial. We begin the exploration of such questions with a standard definition.

\begin{definition} A \textit{generalized arithmetic progression (GAP)} of dimension $k$ (in $\Z$) is a set of the form $$A=\{a+x_1d_1+\cdots+x_kd_k \ : \ 1\leq x_i \leq L_i\},$$ where $a$, $x_i \in \Z$ and $d_i$, $L_i \in \N$. $A$ is called \textit{symmetric} if it takes the form
$$A=\{x_1d_1+\cdots+x_kd_k \ : \ |x_i|\leq L_i\}.$$ 
\end{definition} 

In Section \ref{modelS}, before delving into our more general main results, we give a pleasingly brief, essentially self-contained proof of the following upper bound in the two-dimensional case, with an additional simplifying assumption to suppress any obscuring technical details.  

\begin{theorem} \label{modelI} Suppose $A=\{x_1d_1+x_2d_2 : |x_1|\leq L_1, \ |x_2|\leq L_2 \}\subseteq [-N,N]$ is proper, which is to say $$|A|=(2L_1+1)(2L_2+1).$$ If $L_2d_2\geq L_1d_1$, $d_2$ is prime, and $A$ contains no nonzero squares, then $$|A|\ll N^{5/6}(\log N)^{1/3}.$$
\end{theorem} 

\noindent \textit{Notational Remark.} We shall frequently employ the $\ll$ symbol to mean ``less than a constant times", utilizing subscripts to indicate what the implied constant depends on. In the absence of a subscript, the implied constant is absolute.

Comparing Theorem \ref{modelI} to the trivial and sharp bound of $\sqrt{N}$ in the one-dimensional case, it is natural to ask if one can construct a family of examples of two-dimensional GAPs $A\subset [-N,N]$ with no nonzero squares and $|A|$ significantly greater than $\sqrt{N}$. This sort of lower bound, in this and more general contexts, has thus far proven elusive, and any such examples would be rather interesting. 

\subsection{Main Results}

\subsubsection{Intersective polynomials}

One can generalize the inquiries mentioned thus far in the following way: given a polynomial $h\in \Z[x]$, how large can a symmetric GAP of dimension $k$ be before it is guaranteed to contain a nonzero element in the image of $h$? The following definition encapsulates the largest possible class of polynomials for which a meaningful result of this type is possible. 

\begin{definition} A polynomial $h\in \Z[x]$ is \textit{intersective} if for every $q\in \N$, there exists $r\in \Z$ such that $q \mid h(r)$.
\end{definition}
Intersective polynomials include all polynomials with an integer root, but also include certain polynomials without rational roots such as $(x^3-19)(x^2+x+1)$. In this context it is clear that the intersective condition is necessary, as if $h \in \Z[x]$ has no root modulo $q$, a symmetric GAP $A=\{x_1d_1+\cdots+x_kd_k \ : \ |x_i|\leq L_i\}$ completely misses the image of $h$ whenever $q \mid d_i$ for all $1\leq i\leq k$.  Our first main result is the following.
% Here we employ a very basic Fourier analytic strategy, and our first main result is the following. 

\begin{theorem} \label{main1} Suppose $h\in \Z[x]$ is an intersective polynomial of degree $\ell$ and $A\subseteq [-N,N]$ is a symmetric GAP of dimension $k$. If $A\cap h(\Z) \subseteq \{0\}$, then $$|A|\ll_h N^{1-\big(\ell(2^{\ell+2}+1)^{k-1}\big)^{-1}}(6k)^{2k}\log N.$$ 
\end{theorem}  

\noindent The bound in Theorem \ref{main1} is meaningful for a somewhat small range of dimensions (up to about $k=\log\log N$), but it gives the expected power saving in low dimensions. We include the intersective hypothesis in Theorem \ref{main1} because those are the only polynomials that are compatible with \textit{all} GAPs, but given a fixed GAP, the polynomial only needs to have roots at certain moduli, as indicated by the result's more precise formulation in Section 3.

\subsubsection{$\P$-intersective polynomials}

If we further modify this discussion by restricting the inputs for our polynomials to the primes, which we denote by $\P$, we need to further restrict the class of admissible polynomials, leading to the following definition. 

\begin{definition} A polynomial $h\in \Z[x]$ is \textit{$\P$-intersective} if for every $q\in \N$, there exists $r\in \Z$ such that $q \mid h(r)$ and $(r,q)=1$.
\end{definition} 

$\P$-intersective polynomials include all polynomials with a root at $1$ or $-1$, but again include polynomials without rational roots, in fact the previous example $(x^3-19)(x^2+x+1)$ still qualifies. The necessity of this condition in this context is again clear, and our second main result is the following.

\begin{theorem} \label{main2} Suppose $h\in \Z[x]$ is a $\P$-intersective polynomial of degree $\ell$ and $A\subseteq [-N,N]$ is a symmetric GAP of dimension $k\leq c(\log\log N)/\ell$ for a sufficiently small absolute constant $c>0$. If $A\cap h(\P)\subseteq \{0\}$, then $$|A|\ll_h N^{1-c^k5^{-\ell k}}.$$ 
\end{theorem}

\noindent \textit{Remark relating Theorem \ref{main2} to Linnik's theorem.} As with Theorem \ref{main1}, we give a sharper version of this result in Section 4 in which we fix a GAP and only insist that the polynomial has coprime roots at certain moduli, a result that can be interpreted as a multi-dimensional, polynomial extension of Linnik's theorem. 
Specifically, after some rearrangement of the traditional statement, Linnik's theorem says that if $|a|\leq d$ with $(a,d)=1$ and $p-a \notin \{xd : |x|\leq L\}$ for all $p \in \P$, then $L \ll (Ld)^{1-\epsilon}$ for some $\epsilon>0$ (it is proven that $\epsilon \geq 1/5$ \cite{xyl} and conjectured that $\epsilon=1/2$). Our refinement of Theorem \ref{main2}, the proof of which prominently utilizes a quantitative version of Linnik's theorem, says that if $h\in \Z[x]$ with $\deg(h)=\ell$ has a coprime root modulo $d_i$ for $1\leq i \leq k$, $A=\{x_1d_1+\cdots+x_kd_k : |x_i|\leq L_i\}$ and $A\cap h(\P)\subseteq \{0\}$, then $$\prod_{i=1}^k L_i \ll_h \Big(\sum_{i=1}^k L_id_i\Big)^{1-\epsilon}$$ for some $\epsilon=\epsilon(k,\ell)>0$. In the traditional case of $h(p)=p-a$,  close inspection of the argument reveals that the implied constant can be taken independent of $a$, as required for a true generalization of Linnik's theorem, provided for example that $|a|\leq \min\{d_i\}_{i=1}^k$.

%Before stating our results, we draw parallels between this line of questioning and the far more well-studied area of Diophantine approximation, and in the process we discuss what types of upper bounds we expect to hold.

%If $A\subseteq [-N,N]$ is a symmetric GAP of dimension $k$, then using work of Sanders \cite{Sanders} and Croot, Laba, and Sisask \cite{CLS}, one can show that $A$ contains a Bohr set of low rank and large radius. Combined with a Diophantine approximation result of the last two authors \cite{LR}, who extended Green and Tao's \cite{GT} refinement of an argument of Schmidt \cite{Schmidt} from squares to all intersective polynomials (the second author and Magyar \cite{LM} previously extended it to polynomials without constant term), this yields
%\begin{equation}\label{ht} h\in \Z[x] \text{ intersective }, \ A\cap h(\Z) \subseteq\{0\} \implies |A|\ll N\exp\big(c(k^4-\sqrt{\log N/k})\big),
%\end{equation} where $\ll$ means ``less than a constant times" and $c>0$ and the implied constant depend only on $h$. 

%This result is nontrivial for dimensions up to about $k=(\log N)^{1/9}$, but requires heavy machinery and provides unsatisfying bounds in low dimensions.

\subsection{Motivation from Difference Sets} In a series of papers in the late 1970s, S\'ark\"ozy (\cite{Sark1}, \cite{Sark3}) showed that a set of natural numbers of positive upper density necessarily contains two distinct elements which differ by a perfect square, as well as two elements which differ by one less than a prime number, verifying conjectures of Lov\'asz and Erd\H{o}s, respectively. An extensive literature has been developed on extensions and quantitative improvements of these results, for which the reader may refer to \cite{PSS}, \cite{BPPS}, \cite{Slip}, \cite{Lucier}, \cite{LM2}, \cite{Le},  \cite{HLR}, \cite{Lyall}, \cite{Lucier2}, \cite{Ruz}, \cite{lipan} and \cite{Rice}. For a survey of many of these developments, the reader may refer to Chapter 1 of \cite{thesis}. The following theorems summarize the best known bounds analogous to Theorems \ref{main1} and \ref{main2} for the difference set $A-A=\{a-a' \ : \ a,a'\in A\}\subseteq \Z$. 

\begin{thmA}[\cite{BPPS}, \cite{HLR}, \cite{Lucier}, \cite{thesis}] Suppose $h\in \Z[x]$ is an intersective polynomial of degree $\ell\geq 2$. If $A\subseteq [1,N]$ and $(A-A)\cap h(\N) \subseteq \{0\}$, then 
\begin{equation*}\frac{|A|}{N} \ll_{h,\mu} \begin{cases} (\log N)^{-\mu\log\log\log\log N} & \text{if } \ell=2 \text{ or } h(x)=x^{\ell} \\ \Big( \frac{\log \log N}{\log N} \Big)^{1/(\ell-1)} & \text{else} \end{cases}
\end{equation*}
for any $\mu < 1/\log 3$.
\end{thmA}

\begin{thmB}[\cite{Ruz}, \cite{Rice}] Suppose $h\in \Z[x]$ is a $\P$-intersective polynomial of degree $\ell\geq 2$. If $A\subseteq [1,N]$ and $(A-A)\cap h(\P) \subseteq \{0\}$, then 
\begin{equation}\label{bound}\frac{|A|}{N} \ll_{h,\mu} \begin{cases} e^{-c(\log N)^{1/4}} & \text{if } \ell=1 \\ (\log N)^{-\mu} & \text{if } \ell \geq 2\end{cases}\end{equation} 
for any $\mu<1/(2\ell-2)$, where $c>0$ is an absolute constant. 
\end{thmB} 

The original theorems of S\'ark\"ozy, and the improvements and extensions thereof, are examples of the more general philosophy that sets which are in a sense ``randomly distributed" and avoid certain local biases (such as shifted primes or the image of an intersective polynomial) should always intersect predictably with structured, centrally symmetric sets (such as a difference set). Perhaps the simplest and most explicit example of a highly structured, centrally symmetric set is a low-dimensional symmetric GAP, so the known results for difference sets foreshadow the possibility that similar results for GAPs may be available with more elementary methods and better bounds. In a sense, these questions for GAPs can be interpreted as a ``fantasy model" for the analogous questions for less understood objects like difference sets.

\subsection{Diophantine Approximation and Expected Bounds} We conclude this introductory section with an attempt to contextualize our results as being in some sense parallel to far more well-studied questions in Diophantine approximation, and we attempt to provide some intuition and motivation for what sort of bounds we expect to hold in place of those in Theorems \ref{main1} and \ref{main2}. We speak rather informally, in particular the referenced results are somewhat roughly approximated, and we make light of the distinction between an interval of integers $[-N,N]$ and a finite cyclic group $\Z_N:=\Z/N\Z$. 

We continue the theme of the previous section by considering another related example of a highly structured, centrally symmetric set. Given a natural number $N$, a set of frequencies $\Gamma= \{\xi_1,\dots,\xi_k\}\subseteq \Z_N$, and $\epsilon>0$, the Bohr set $B=B(\Gamma, \epsilon)\subseteq \Z_N$ is defined by $$B=\{x\in \Z_N \ : \ \norm{x\xi_j/N}<\epsilon \text{ for } 1\leq j \leq k\},$$ where $\norm{\cdot}$ denotes the distance to the nearest integer. We refer to $k$ as the \textit{rank} of $B$. If we pose the analogous questions for Bohr sets that we have been considering for symmetric GAPs, then we wade into questions of Diophantine approximation. An example of such an analog is the following result of Green and Tao, a refinement of an argument due to Schmidt \cite{Schmidt}.

\begin{thmC}[Proposition A.2 in \cite{GT}]
Given $\alpha_1,\dots,\alpha_k\in \R$ and $N\in\N$, there exists an integer $1\leq n\leq N$ such that \[\|n^2\alpha_j\|\ll kN^{-c/k^2} \text{ \ for all \ }1\leq j \leq k.\] 
\end{thmC}

\noindent The second and third authors further adapted this argument to include all intersective polynomials. 

\begin{thmD}[Theorem 1 in \cite{LR}] Given $\alpha_1, \dots, \alpha_k \in \R$, an intersective polynomial $h\in \Z[x]$ of degree $\ell$, and $N\in \N$, there exists an integer $1\leq n \leq N$ with $h(n)\neq 0$ and $$\norm{h(n)\alpha_j} \ll_h kN^{-c^{\ell}/k^2} \text{ for all } 1\leq j \leq k,$$ where $c>0$ is an absolute constant.
\end{thmD}

In particular, Theorem D says that a Bohr set $B\subseteq \Z_N$ of rank $k$ contains a nonzero element in the image of $h$ provided the radius is at least a large constant times $kN^{-c^{\ell}/k^2}$. Probabilistically, we expect a Bohr set $B\subseteq \Z_N$ of rank $k$ and radius $\epsilon$ to have size about $\epsilon^kN$, suggesting that if $B\cap h(\N) \subseteq \{0\}$ and $k$ is much smaller than $\sqrt{\log N}$, then \begin{equation}\label{BB} |B|\ll_h N^{1-c^{\ell}/k}.\end{equation} It is conjectured that the $k^2$ in the exponent in Theorem B can be replaced with $k^{1+o(1)}$, suggesting  bounds even stronger than (\ref{BB}). 

 Bohr sets and symmetric GAPs can be thought of as similar, intimately related objects. Like a symmetric GAP, a Bohr set is a centrally symmetric set with a great deal of structure mimicking that of a subgroup. It is a standard fact (see \cite{TV} for example) that a Bohr set $B\subseteq \Z_N$ of rank $k$ and radius $\epsilon$ contains a symmetric GAP $A\subseteq \Z_N$ (defined analogously to a GAP in $\Z$) of dimension $k$ with $|A|\geq (\epsilon/k)^kN$. Conversely, if $A\subseteq \Z_N$ is a symmetric GAP of dimension $k$ with $|A|=\delta N$, then  work of Sanders \cite{Sanders} and Croot, Laba, and Sisask \cite{CLS} shows that $A$ contains a Bohr set of rank roughly $\log(1/\delta)+k^4$ and radius roughly $2^{-k}$. The latter conclusion combines with Theorem D to yield 
 \begin{equation}\label{ht} h\in \Z[x] \text{ intersective }, \ A\cap h(\Z) \subseteq\{0\} \implies |A|\ll N\exp\big(c(k^4-\sqrt{\log N/k})\big).
\end{equation}

Compared to Theorem \ref{main1}, this result is nontrivial for a much larger range of dimensions,  up to about $k=(\log N)^{1/9}$, but requires much heavier machinery and provides unsatisfying bounds in low dimensions. In contrast, our proofs of Theorems \ref{main1} and \ref{main2} rely only on very basic Fourier analytic methods and standard number theory and circle method facts.

Due to the aforementioned intimate connections between symmetric GAPs and Bohr sets, we believe that the analysis of polynomial configurations in Bohr sets via Diophantine approximation provides meaningful insight into the analogous questions for GAPs. The structure of a Bohr set is richer than that of a GAP, so for now we tread lightly with our conjectures, but these heuristics for bounds like (\ref{BB}) (or beyond) suggest that the doubly exponential dependence on the dimension $k$  in Theorems \ref{main1} and \ref{main2} is far from the truth. We return to the consideration of improved bounds in Section \ref{special}.

\begin{center} \textbf{Funding} \end{center}
The first author was partially supported by National Science Foundation Grant DMS-1001111. The second author was partially supported by Simons Foundation Collaboration Grant for Mathematicians 245792.
 
  \begin{center} \textbf{Acknowledgements} \end{center}
The authors would like to thank Paul Pollack for helpful comments, as well as the referees for their helpful corrections and recommendations. 
  
\section{Preliminaries} 

\subsection{Fourier Analysis on $\Z/d\Z$} For $d\in \N$ and a complex-valued function $f$ on $\Z_d:=\Z/d\Z$, we utilize the unnormalized, discrete Fourier transform $\widehat{f}:\Z_d\to \C$, defined by 
\[\widehat{f}(t)=\sum_{x\in \Z_d}f(x)e^{-2\pi i xt/d}.\]
\noindent In this discrete setting, the Fourier inversion formula 
\begin{equation}\label{inv} f(x)=\frac{1}{d}\sum_{t\in \Z_d}\widehat{f}(t)e^{2\pi i xt/d}\end{equation} 
is a simple consequence of the orthogonality relation 
\begin{equation}\label{orth} \frac{1}{d}\sum_{t\in \Z_d} e^{2\pi i x t/d} = \begin{cases} 1 & \text{if } x=0 \\ 0 & \text{if } x\in \Z_d\setminus \{0\} \end{cases}.
\end{equation}
For two functions $f,g:\Z_d \to \C$, we define the convolution $f*g:\Z_d \to \C$ by 
\[f*g(x)=\sum_{y\in \Z_d}f(y)g(x-y),\]
and the property 
\begin{equation} \label{conv} \widehat{f*g}(t)=\widehat{f}(t)\widehat{g}(t)
\end{equation}
also follows quickly from (\ref{orth}).

\subsection{Exponential Sum Estimates} The key to the argument, as is often the case, is to take advantage of cancellation in exponential sums over the image of a polynomial away from rationals with small denominator. First, we include a proof of the standard estimate in the special case of quadratic polynomials, with a view toward a self-contained proof of Theorem \ref{modelI}.

\begin{lemma} \label{gauss} If $t, a_1,a_2 \in \Z$ and $n,d\in \N$, then 
$$\Big|\sum_{m=1}^ne^{2\pi i (a_1m+a_2m^2)t/d} \Big| \leq \Big(2n^2/d'+7(n+d')\log d'\Big)^{1/2},$$
where $d'=d/(2a_2t,d)$. 
\end{lemma}

\begin{proof} Using a change of variables ($m=s+h$), we see
\begin{align*}\Big|\sum_{m=1}^ne^{2\pi i (a_1m+a_2m^2)t/d} \Big|^2 &=\sum_{m,s=1}^ne^{2\pi i \bigl((a_1m+a_2m^2)-(a_1s+a_2s^2)\bigr)t/d}=\sum_{s=1}^{n}\sum_{h=1-s}^{n-s}e^{2\pi i\bigl(a_2(2sh+h^2)+a_1h\bigr)t/d} \\\\ & = n + \sum_{h=1}^{n-1} \sum_{s=1}^{n-h}e^{2\pi i\bigl(a_2(2sh+h^2)+a_1h\bigr)t/d}+\sum_{h=1-n}^{-1} \sum_{s=1-h}^{n}e^{2\pi i\bigl(a_2(2sh+h^2)+a_1h\bigr)t/d} \\\\ &\leq n +2\sum_{h=1}^{n-1}\Big|\sum_{s=1}^{n-h}e^{2\pi ish(2a_2t)/d}\Big|.
\end{align*}
\noindent Writing $2a_2t/d=t'/d'$ with $(t',d')=1$, letting $\norm{\cdot}$ denote distance to the nearest integer, and applying the geometric series formula, we have
\begin{align*}\sum_{h=1}^{n-1}\Big|\sum_{s=1}^{n-h}e^{2\pi isht'/d'}\Big|\leq \sum_{h=1}^{n-1} \min \{n, (2\norm{ht'/d'})^{-1}\} \leq \Big(\frac{n}{d'}+1\Big)\Big(n+\sum_{h=1}^{d'-1} \frac{d'}{h}\Big) \leq n^2/d'+3(n+d')\log d',
\end{align*}
provided $d'>1$ (the lemma is trivial if $d'=1$), and the result follows. \end{proof}

\noindent For more general polynomials, we first invoke the following standard bound obtained from Weyl differencing.

\begin{lemma} [Proposition 8.2 in \cite{IK}] \label{IKlem} If $h(x)=\alpha x^{\ell}+\cdots \in \R[x]$, then $$\Big| \sum_{m=1}^n e^{2\pi i h(m)}\Big| \ll n \Big(n^{-\ell} \sum_{-n<m_1,\dots, m_{\ell-1}<n} \min \{ n, \norm{\alpha \ell! m_1\dots m_{\ell-1}}^{-1}\}\Big)^{2^{1-\ell}}.$$
\end{lemma}

\noindent From Lemma \ref{IKlem}, we deduce the following refined version of Weyl's Inequality by exploiting the second moment, as opposed to the maximum value, of an appropriate divisor function.

\begin{lemma} \label{weylI}  Suppose $h(x)=a_0+a_1x+\cdots+a_{\ell}x^{\ell}$ with $a_i \in \R$ and $a_{\ell} \in \N$. If $(t,d)=1$ and $|\alpha-t/d|<d^{-2}$, then 
$$ \Big|\sum_{m=1}^n e^{2\pi i h(m)\alpha} \Big| \ll_{\ell} n \Big(a_{\ell}\log^{\ell^2}(a_{\ell}dn)(1/d+1/n+d/a_{\ell}n^{\ell}) \Big)^{2^{-\ell}}.$$
\end{lemma} 
 
\begin{proof}We begin by recalling that if $d_j(m)=\Big|\{(a_1,\dots,a_j): a_i \in \N, \ a_1\cdots a_j=m\}\Big|,$ then we know from standard estimates (see \cite{IK} for example) that
\begin{equation} \label{div} \sum_{m=1}^M d_j(m)^2 \ll_j M\log^{j^2-1}(M).
\end{equation}
Applying Lemma \ref{IKlem}, Cauchy-Schwarz, and (\ref{div}), we have
\begin{align*} \Big| \sum_{m=1}^n e^{2\pi i h(m)\alpha}\Big| &\ll n \Big(n^{-\ell} \sum_{-n<m_1,\dots, m_{\ell-1}<n} \min \{ n, \norm{\alpha \ell!a_{\ell} m_1\dots m_{\ell-1}}^{-1}\}\Big)^{2^{1-\ell}} \\\\ & \ll_{\ell} n^{1-\ell2^{1-\ell}} \Big(n^{\ell-1} + \sum_{1\leq m \leq n^{\ell-1}} d_{\ell-1}(m)\min\{n, \norm{\alpha \ell! a_{\ell}m}^{-1} \} \Big)^{2^{1-\ell}} \\\\ & \leq n^{1-\ell2^{1-\ell}} \Big(n^{\ell-1} + \bigl(\sum_{1\leq m \leq n^{\ell-1}} d_{\ell-1}(m)^2 \bigr)^{1/2} \bigl(\sum_{1\leq m \leq n^{\ell-1}} \min\{n, \norm{\alpha \ell! a_{\ell}m}^{-1} \}^2  \bigr)^{1/2} \Big)^{2^{1-\ell}} \\\\ & \ll_{\ell} n^{1-\ell2^{1-\ell}} \Big(n^{\ell-1} + n^{\ell/2}\log^{\frac{\ell^2-1}{2}}(n) \bigl(\sum_{1\leq m \leq n^{\ell-1}} \min\{n, \norm{\alpha \ell! a_{\ell}m}^{-1}\}   \bigr)^{1/2}\Big)^{2^{1-\ell}}.
\end{align*}
By Lemma 2.2 in \cite{vaughan}, we know that if $|\alpha-t/d|<d^{-2}$ with $(t,d)=1$, then
$$ \sum_{1\leq m \leq n^{\ell-1}} \min\{n, \norm{\alpha \ell!a_{\ell}m}^{-1}\} \ll a_{\ell}n^{\ell}\log(a_{\ell}dn)\Big(1/d+1/n+d/a_{\ell}n^{\ell} \Big),$$ and the result follows.
\end{proof}

\noindent For the primes we invoke the following analog of Weyl's inequality, a nominal generalization of Theorem 4.1 in \cite{lipan}, extracted in a slightly more precise way as compared to Lemma 12 in \cite{Rice}.

\begin{lemma}\label{min} Suppose $h(x)=a_0+a_1x+\cdots+a_{\ell}x^{\ell} \in \Z[x]$ with $a_{\ell}>0$,  and let $L=64\ell^24^{\ell}$. If $U\geq \log n$,
$a_{\ell} \geq C\Big(|a_{\ell-1}| + \cdots +|a_0|\Big),$ and $q,|r|,a_{\ell} \leq U^{\ell}$, then 
\begin{equation*}\sum_{\substack{m=1 \\ qm+r \in \P}}^n \log(qm+r)e^{2\pi i h(m) \alpha} \ll_C \frac{n}{U}+U^Ln^{1-4^{-\ell}}
\end{equation*} 
provided \begin{equation*} |\alpha -t/d| < d^{-2} \quad \text{for some} \quad U^{L} \leq d \leq h(n)/U^{L} \quad \text{and} \quad (t,d)=1.
\end{equation*}
\end{lemma}
\noindent Because Lemma \ref{min} does not give any information for very small denominators, we require the following high-moment estimate, which follows in particular from a solution to the Waring-Goldbach problem (see \cite{Hua} for example).

\begin{lemma} \label{moment} If $h\in \Z[x]$ with $\deg(h)=\ell$ and $s>2^{\ell}$, then
$$\sum_{t\in \Z_d} \Big| \sum_{\substack{m=1 \\ qm+a \in \P}}^n \log(qm+a)e^{2\pi i h(m) t/d} \Big|^s \ll_h dn^{s-\ell} + n^s.$$ 
\end{lemma}

\subsection{Primes in Arithmetic Progressions} In order to obtain the bound purported in Theorem \ref{main2}, we need a quantitative strengthening of Linnik's theorem, which says that if $(a,q)=1$, then there are plenty of primes less than a parameter $x$ which are congruent to $a$ modulo $q$, even if $q$ is as large as a small power of $x$. This exceeds, at the expense of a sharp asymptotic formula, the allowed modulus size for usual prime number theorems for arithmetic progressions.

\begin{lemma}[Quantitative Linnik's theorem, Corollary 18.8 in \cite{IK}] \label{qlin} If $(a,q)=1$ and $q \leq cx^{c}$ for a sufficiently small constant $c>0$, then $$\psi(x,a,q):=\sum_{\substack{p\in \P\cap[1,x] \\ p \equiv a \textnormal{ mod }q}} \log p \gg \frac{x}{\phi(q)\sqrt{q}},$$ where $\phi$ is the Euler totient function. 
\end{lemma}
\noindent The requisite tools are now in place for us to proceed with proving Theorems \ref{modelI}, \ref{main1}, and \ref{main2}.

\section{Intersective Polynomials}

\subsection{The Model Case} \label{modelS} In this section, we prove the simplest nontrivial extension of the initial trivial observation made in the introduction, which exposes  the fundamental ideas of our general argument. 
\begin{proof}[Proof of Theorem \ref{modelI}] Suppose $A=\{x_1d_1+x_2d_2 : |x_1|\leq L_1, \ |x_2|\leq L_2 \}\subseteq [-N,N]$ is proper, $L_2d_2\geq L_1d_1$,  and $d_2 \in \P$. Let $n=(L_2d_2/4)^{1/2}$.
\noindent Suppose  further that $A$ contains no nonzero squares, which crucially means that no sufficiently small square is congruent to any sufficiently small multiple of $d_1$ modulo $d_2$.

\noindent Let $A_1=\{xd_1 : |x|\leq L_1/4\}\subseteq \Z_{d_2}$. We note that, by properness and the fact that $L_1d_1\leq L_2d_2$, all of these multiples of $d_1$ reduce to distinct residues modulo $d_2$, and in particular $|A_1|\gg L_1$. We define $f$ as a function on $\Z_{d_2}$ by $$f(x)=|A_1|^{-1}1_{A_1}*1_{A_1}(x).$$ Clearly $f$ is a nonnegative function supported on $\{xd_1 : |x|\leq L_1/2\}\subseteq \Z_{d_2}$ and $f(x)\leq f(0) = 1$. The symmetry of $A_1$ makes $\widehat{1_{A_1}}$ a real-valued function, and by (\ref{conv}) we know that $\widehat{f}(t)=|A_1|^{-1}\widehat{1_{A_1}}(t)^2$, therefore \[0\leq \widehat{f}(t) \leq \widehat{f}(0)=|A_1|.\]

\noindent We note that if $f(m^2)>0$ for some $1\leq m \leq n$, then there exists $x$ with $|x|\leq L_1/2$ and $xd_1\equiv m^2 \mod d_2$. In other words, $xd_1-m^2=\ell d_2$ and $$|\ell|\leq (L_1d_1/2+m^2)/d_2\leq (L_2d_2/2+L_2d_2/2)/d_2=L_2,$$ hence $m^2 =xd_1-\ell d_2\in A$. 
\noindent Therefore, since we assumed $A$ contained no nonzero squares, it follows from the orthogonality relation (\ref{orth}) that 
\begin{equation*} d_2\sum_{1\leq m \leq n}f(m^2)=\sum_{t\in \Z_{d_2}}\widehat{f}(t)W(t)=0,
\end{equation*} 
where $$W(t)=\sum_{m=1}^n e^{2\pi i m^2t/d_2}.$$ Combined with the positivity of $\widehat{f}$, this yields 
\begin{equation} \label{sq1} \sum_{t\neq 0} \widehat{f}(t)|W(t)| \geq \widehat{f}(0)W(0) \gg L_1(L_2d_2)^{1/2}.
\end{equation} 
\noindent  Since $d_2\in \P$ and $n\leq d_2$, as otherwise $d_2^2\in A$,  we have from Lemma \ref{gauss} that 
\begin{equation*} |W(t)| \ll (d_2 \log d_2)^{1/2} \text{ for all } t\neq 0,
\end{equation*}
which by the Fourier inversion formula (\ref{inv}) and the fact that $f(0)=1$ gives
\begin{equation}\label{sq2} \sum_{t\neq 0} \widehat{f}(t)|W(t)| \ll (d_2 \log d_2)^{1/2} \sum_{t\in \Z_{d_2}} \widehat{f}(t) = (d_2 \log d_2)^{1/2} d_2f(0) \leq d_2(d_2 \log N)^{1/2}.
\end{equation}
\noindent Combining (\ref{sq1}) and (\ref{sq2}), we have $$L_1(L_2d_2)^{1/2} \ll d_2(d_2 \log N)^{1/2},$$ 
which can be manipulated to yield 
$$|A|\ll L_1L_2 \ll \frac{L_2d_2}{L_2^{1/2}}(\log N)^{1/2}\leq \frac{N}{L_2^{1/2}}(\log N)^{1/2}.$$
Further, we know from our discussion of the one dimensional case that $L_1\ll N^{1/2}$, and hence $$|A|\ll \min \{N^{1/2}L_2, \frac{N}{L_2^{1/2}}(\log N)^{1/2}\}. $$ The quantity on the right hand side is maximized when $L_2=(N\log N)^{1/3}$, and the theorem follows.
\end{proof} 

\subsection{The General Case} In this section we prove Theorem \ref{main1} in a more flexible form, and we begin by making part of this flexibility precise with a definition.  
\begin{definition}For $M>0$, we say that a GAP $A$ is $M$-\textit{proper} if each element of $A$ is represented in at most $M$ ways. In the case of a symmetric progression $A=\{x_1d_1+\cdots+x_kd_k : |x_i|\leq L_i\}$, this means $$\Big|\Big\{(x_1,\dots, x_k)\in \Z^k: |x_i|\leq L_i, \ x_1d_1+\cdots+x_kd_k=x\Big\}\Big| \leq M \text{ for all } x\in \Z.$$ 
\end{definition}

\noindent  We now deduce Theorem \ref{main1} from the following.
\begin{theorem} \label{M1} Suppose $A=\{x_1d_1 + \cdots + x_kd_k : \ |x_i|\leq L_i \} \subseteq [-N,N]$ is an $M$-proper, symmetric GAP with $L_1 \leq \cdots \leq L_k$, and suppose $h\in \Z[x]$ is a polynomial of degree $\ell$ which has a root modulo $\gcd(d_i,\dots,d_k)$ for all $1\leq i \leq k$. If $A\cap h(\Z) \subseteq \{0\}$, then $$\prod_{i=1}^k L_i\ll_h  N^{1-\big(\ell(2^{\ell+2}+1)^{k-1}\big)^{-1}}(4k)^kM\log N.$$ 
\end{theorem} 

\begin{proof}[Proof that Theorem \ref{M1} implies Theorem \ref{main1}]
Suppose $A=\{x_1d_1 + \cdots + x_kd_k : \ |x_i|\leq L_i \} \subseteq [-N,N]$ and let $$M=\max_{x\in \Z}\Big|\Big\{(x_1,\dots, x_k)\in \Z^k: |x_i|\leq L_i, \ x_1d_1+\cdots+x_kd_k=x\Big\}\Big| $$ Fixing an $x\in \Z$ at which this maximum is attained, we find by taking differences of the $M$ representations of $x$ that there are at least $M$ representations 
\begin{align*} y_1^{(1)}+&\cdots+y_k^{(1)}=0 \\ y_1^{(2)}+&\cdots+y_k^{(2)}=0 \\ &\cdots \\ y_1^{(m)}+&\cdots+y_k^{(m)}=0
\end{align*} with $|y_i|\leq 2L_i$. In particular, there are at least $M$ representations of every element of $A$ inside the larger GAP $\{x_1d_1+\cdots+x_kd_k : |x_i|\leq 3L_i\}$, and hence $$M|A| \leq \prod_{i=1}^k (6L_i+1).$$ Theorem \ref{main1} then follows immediately from Theorem \ref{M1}
\end{proof}

\noindent We deduce Theorem \ref{M1} from the following lemma, which states that, provided there are no obvious local obstructions, a GAP contains a nonzero element in the image of a given polynomial as long as it is sufficiently ``wide" in each direction. Theorem \ref{M1} then follows from an iteration that is responsible for the doubly exponential dependence on the dimension $k$ in the final result.
 
\begin{lemma} \label{itn} Suppose $A=\{x_1d_1 + \cdots + x_kd_k : \ |x_i|\leq L_i \} \subseteq [-N,N]$ is an $M$-proper, symmetric GAP with $L_1\cdots L_k=\delta N$, and suppose $h\in \Z[x]$ is a polynomial of degree $\ell$ which has a root modulo $(d_1,\dots,d_k)$. If
\begin{equation}\label{fat1} \min\{L_i\}_{i=1}^k \geq C\Big(M(4k)^{k}\delta^{-1}\log N\Big)^{2^{\ell+2}} \end{equation} 
for a sufficiently large constant $C$ depending only on $h$, then $A$ contains a nonzero element of $h(\Z)$.
\end{lemma}

\subsubsection{Proof of Theorem \ref{M1}} Fix an $M$-proper GAP 
\[A_0=\{x_1d_1 + \cdots + x_kd_k : \ |x_i|\leq L_i \}\subseteq [-N,N]\] 
with $L_1\cdots L_k=\delta_0N$, ordered such that $L_1\leq \cdots \leq L_k$, and fix a polynomial $h \in \Z[x]$ of degree $\ell$ such that $h$ has a root modulo $(d_i,\dots, d_k)$ for all $1\leq i \leq k$ and $A\cap h(\Z) \subseteq \{0\}$. 
By Lemma \ref{itn}, we know \[L_1\leq C\Big(M(4k)^{k}\delta_0^{-1}\log N\Big)^{2^{\ell+2}},\] 
and removing the first dimension yields \[A_1=\{x_2d_2 + \cdots + x_kd_k : \ |x_i|\leq L_i \}\] 
with $L_2\cdots L_k=\delta_1N$ and \[\delta_1=\delta_0/L_1 \geq C\Big(M(4k)^{k}\delta_0^{-1}\log N\Big)^{-(2^{\ell+2}+1)}.\] 
Iterating this process yields GAPs \[A_j=\{x_{j+1}d_{j+1} + \cdots + x_kd_k : \ |x_i|\leq L_i \}\] with $L_{j+1}\cdots L_k=\delta_j N$ and \[\delta_j \geq \Big( C M(4k)^{k}\delta_0^{-1}\log N\Big)^{-(2^{\ell+2}+1)^j}\] for $1\leq j \leq  k-1$. Further, if  $r \in [0,d_k)$ is a root of $h$ modulo $d_k$, then one of \[\{h(r), h(r+d_k), \dots, h(r+(\ell+1)d_k)\},\] call it $h(n)$, is nonzero, and hence $h(n)\notin A_k$. Therefore, since $d_k \mid h(n)$, it must be the case that $$L_kd_k \leq |h(n)| \ll_h d_k^{\ell},$$ and hence $$(\delta_{k-1} N)^{1+1/(\ell-1)} = L_k^{1+1/(\ell-1)}\ll_h L_kd_k \leq N.$$ This implies that $\delta_{k-1} \ll _h N^{-1/\ell}$, hence $$\Big( C M(4k)^{k}\delta_0^{-1}\log N\Big)^{(2^{\ell+2}+1)^{k-1}}\geq N^{1/\ell},$$ and the theorem follows.
\qed

\subsubsection{Proof of Lemma \ref{itn}} Fix an $M$-proper GAP \[A=\{x_1D_1 + \cdots + x_kD_k : \ |x_i|\leq L_i \}\subseteq[-N,N],\] let $\delta=L_1\cdots L_k/N$, let $q=(D_1,\dots, D_k)$, and fix a polynomial $h\in\Z[x]$ of degree $\ell$ with a root modulo $q$, which by symmetry we can assume has positive leading coefficient. To show that $A$ contains a nonzero element in the image of $h$, it suffices to show that \[A'=\{x_1d_1 + \cdots + x_kd_k : \ |x_i|\leq L_i \}\] contains an nonzero element in the image of $h_q(x)=h(r+qx)/q\in \Z[x]$, where $r\in[0,d)$ is a root of $h$ modulo $q$ and $d_i=D_i/q$. 
\noindent Let \[S=\{m\in \N : 0<h_q(m)<L_kd_k/2\}\] and $n=\Big(\dfrac{L_kd_k}{2q^{\ell-1}b}\Big)^{1/\ell}$, where $b$ is the leading coefficient of $h$, noting that 
\begin{equation}\label{symdif}\Big| S \triangle [1,n]\Big| \ll_h 1.\end{equation}

\noindent For $1\leq i \leq k-1$, assume without loss of generality that $L_id_i \leq L_kd_k$, let \[A_i=\{xd_i: |x|\leq L_i/4k\}\subseteq \Z,\] and define $g_i,f_i$ as functions on $\Z_{d_k}$ by $$g_i(x)=\sum_{|y|\leq L_i/4k} 1_{yd_i}(x) \quad \text{and} \quad f_i(x)=|A_i|^{-1}g_i*g_i(x).$$ 
As in the model case, $f_i$ is supported on $\{xd_i: |x|\leq L_i/2k\}\subseteq \Z_{d_k}$ and \[0\leq \widehat{f_i}(t)\leq \widehat{f_i}(0)=|A_i|.\] 
Letting $M_i=\max(f_i)$, we see that $\prod_{i=1}^{k-1}M_i\leq M$ by $M$-properness, and  $f_1*\cdots*f_{k-1}(0)$ is bounded above by \begin{equation}\label{Mprop} \Big|\Big\{(x_1d_1,\dots,x_{k-1}d_{k-1})\in \Z_{d_k}^{k-1} \ : \ |x_i|\leq L_i/2k, \ x_1d_1+\cdots+x_kd_k \equiv 0 \mod d_k\Big\}\Big| \prod_{i=1}^{k-1}M_i \leq M^2.\end{equation}  Further, if $f_1*\cdots*f_{k-1}(h_q(m))>0$ for some $m \in S$, then $h_q(m) \in A'$, so we need only show
\begin{equation*}d_k\sum_{m\in S} f_1*\cdots*f_{k-1}(h_q(m))=\sum_{t\in \Z_{d_k}}\widehat{f_1}(t)\cdots\widehat{f_{k-1}}(t)W(t)>0,
\end{equation*}
where the equality follows from (\ref{orth}) and $$W(t)=\sum_{m\in S}e^{2\pi i h_q(m)t/d_k},$$ 
for which it suffices to show
\begin{equation}\label{suff} \sum_{t\neq 0}\widehat{f_1}(t)\cdots\widehat{f_{k-1}}(t)|W(t)| < \widehat{f_1}(0)\cdots\widehat{f_{k-1}}(0)W(0)=(n+O_h(1))\prod_{i=1}^{k-1}|A_i|,
\end{equation} where the estimate on $W(0)$ follows from (\ref{symdif}). Noting that the leading coefficient of $h_q$ is at most $q^{\ell}$ times the leading coefficient of $h$, and that $q\leq \delta^{-1}$ by definition, we have from Lemma \ref{weylI} and (\ref{symdif})  that if $(t,d_k)=D$,  then \begin{equation}\label{gauss2} |W(t)| \ll_h  n\Big(\delta^{-\ell}\log^{\ell^2} N(D/d_k+1/n+1/DL_k)\Big)^{2^{-\ell}}.
\end{equation}  We also have by (\ref{Mprop}) and the Fourier inversion formula that 
\begin{equation}\label{finv} \sum_{t\in\Z_{d_k}} \widehat{f_1}(t)\cdots\widehat{f_{k-1}}(t)= d_k f_1*\cdots*f_{k-1}(0)\leq M^2d_k. \end{equation}
Further, we know that  
\begin{equation}\label{alphaB} d_k \leq N/L_k = \Big(\prod_{i=1}^kL_i\Big)/\delta L_k\leq \Big(\prod_{i=1}^{k-1}|A_i|\Big)(4k)^k/\delta. 
\end{equation}
By (\ref{gauss2}), (\ref{finv}), and (\ref{alphaB}), we see that for any $X>0$ 
\begin{equation} \label{A2} \sum_{(t,d_k) \leq d_k/X} \widehat{f_1}(t)\cdots\widehat{f_{k-1}}(t)|W(t)| \ll_h n\Big(\prod_{i=1}^{k-1}|A_i|\Big) \Big(\delta^{-\ell}\log^{\ell^2} N(1/X+1/n+1/L_k)\Big)^{2^{-\ell}}M^2(4k)^k/\delta
\end{equation} We see that if $(t,d_k)=D>d_k/X$, and $d_k/(t,d_k) \nmid d_j$, then
\begin{equation*}\widehat{f_j}(t) = |A_j|^{-1}\Big| \sum_{|x|\leq L_j/4k} e^{2\pi i xd_j t/d_k} \Big|^2 \leq |A_j|^{-1} \norm{d_j t/d_k}^{-2}\leq X^2/|A_j|,
\end{equation*}
where $\norm{\cdot}$ denotes the distance to the nearest integer. Therefore, \begin{equation}\label{canc}\sum_{(t,d_k) =D}\widehat{f_1}(t)\cdots\widehat{f_{k-1}}(t) \leq X^2|A_j|^{-1} \sum_{D \mid t} \prod_{i\neq j}\widehat{f_i}(t)\leq X^3\Big(\prod_{i=1}^{k-1}|A_i|\Big)/|A_j|^2.
\end{equation}
There are trivially at most $X$ divisors of $d_k$ satisfying $D>d_k/X$, and since $(d_1,\dots,d_k)=1$, (\ref{canc}) implies
\begin{equation}\label{A3} \sum_{(t,d_k) > d_k/X} \widehat{f_1}(t)\cdots\widehat{f_{k-1}}(t)|W(t)| \ll_h n\Big(\prod_{i=1}^{k-1}|A_i|\Big) X^4/ \min\{|A_i|\}^2.
\end{equation}
Setting $X= CM^{2^{\ell+1}}(4k)^{2^{\ell}k}(\log N)^{\ell^2}\delta^{-(2^{\ell}+\ell)}$ for a sufficiently large constant $C$ depending on $h$, we see that if $A$ satisfies (\ref{fat1}), then $$\min\{|A_i|\}_{i=1}^{k-1} \geq \min\{L_i\}_{i=1}^{k-1}/2k \geq C'X^{2}$$ for a sufficiently large constant $C'$, provided the constant in (\ref{fat1}) is sufficiently large with respect to the constant defining $X$. Combining (\ref{A2}) and (\ref{A3}), we see that the left-hand side of (\ref{suff}) is dominated by a small constant times the right-hand side of (\ref{suff}), and the lemma follows. \qed
 
\section{$\P$-intersective Polynomials} Analogous to Section 3, we deduce Theorem \ref{main2} from the following result.
\begin{theorem}\label{M2}  Suppose $A=\{x_1d_1 + \cdots + x_kd_k : \ |x_i|\leq L_i \} \subseteq [-N,N]$ is an $M$-proper, symmetric GAP with $L_1\leq L_2 \leq \cdots \leq L_k$, and suppose $h\in \Z[x]$ is a polynomial of degree $\ell$ which has a root modulo $(d_i,\dots,d_k)$ that is coprime to $(d_i,\dots, d_k)$ for all $1\leq i \leq k$. If $A\cap h(\P)\subseteq\{0\}$, then $$\prod_{i=1}^k L_i\ll N^{1-c\ell^{-1}(200\ell^24^{\ell})^{1-k}} (4k)^k M\log N $$ for an absolute constant $c>0$, where the implied constant depends only on $h$.
\end{theorem}

\noindent Virtually identical to the deduction of Theorem \ref{M1} from Lemma \ref{itn}, Theorem \ref{M2} follows from Linnik's theorem and the following lemma.

\begin{lemma}\label{itn2} Suppose $A=\{x_1d_1 + \cdots + x_kd_k : \ |x_i|\leq L_i \} \subseteq [-N,N]$ is an $M$-proper, symmetric GAP with $L_1\cdots L_k=\delta N$, and suppose $h\in \Z[x]$ is a polynomial of degree $\ell$ with a root modulo $(d_1,\dots, d_k)$ that is coprime to $(d_1,\dots, d_k)$.  If $\delta \geq CN^{-c/\ell}$ and
\begin{equation}\label{minL2} \min\{L_i\}_{i=1}^k \geq \Big(CM(4k)^k\delta^{-1}\log N\Big)^{200\ell^2 4^{\ell}-1} \end{equation} 
for sufficiently large and small constants $C$ and $c$, respectively, then $0\neq h(p) \in A$ for some $p\in \P$.
\end{lemma}

\begin{proof}Fix an $M$-proper GAP \[A=\{x_1D_1 + \cdots + x_kD_k : \ |x_i|\leq L_i \}\subseteq[-N,N],\] let $q=(D_1,\dots, D_k)$, and fix a polynomial $h\in\Z[x]$ of degree $\ell$ with a root $r\in[0,q)$ modulo $q$ such that $(r,q)=1$, which by symmetry we can assume has positive leading coefficient.  To show that $A$ contains a nonzero element of the form $h(p)$ with $p\in \P$, it suffices to show that \[A'=\{x_1d_1 + \cdots + x_kd_k : \ |x_i|\leq L_i \}\] contains a nonzero element of the form $h_q(m)=h(qm+r)/q$ with $qm+r \in \P$, where $d_i=D_i/q$. 
Let \[\Lambda=\{m\in \N : qm+r \in \P, \ 0<h_q(m)<L_kd_k/2\}\] and let $n=\Big(\dfrac{L_kd_k}{2q^{\ell-1}b}\Big)^{1/\ell}$, where $b$ is the leading coefficient of $h$. 

\noindent For $1\leq i \leq k-1$, assume without loss of generality that $L_id_i \leq L_kd_k$, let \[A_i=\{xd_i: |x|\leq L_i/4k\}\subseteq \Z,\] and define $g_i,f_i$ as functions on $\Z_{d_k}$ by $$g_i(x)=\sum_{|y|\leq L_i/4k} 1_{yd_i}(x) \quad \text{and} \quad f_i(x)=|A_i|^{-1}g_i*g_i(x).$$ As in the proof of Lemma \ref{itn}, $f_i$ is supported on $\{xd_i: |x|\leq L_i/2k\}\subseteq \Z_{d_k}$, \begin{equation} \label{Mprop2} f_1*\cdots*f_{k-1}(0)\leq M^2\end{equation} by $M$-properness, and \[0\leq \widehat{f_i}(t)\leq \widehat{f_i}(0)=|A_i|.\] 
Further, if $f_1*\cdots*f_{k-1}(h_q(m))>0$ for some $m \in \Lambda$, then $0\neq h(p)\in A$ for some $p\in \P$, so we need only show
\begin{equation*}d_k\sum_{m\in \Lambda} \log(qm+r) f_1*\cdots*f_{k-1}(h_q(m))=\sum_{t\in \Z_{d_k}}\widehat{f_1}(t)\cdots\widehat{f_{k-1}}(t)V(t)>0, 
\end{equation*}
where the equality follows from (\ref{orth}) and  $$V(t)=\sum_{m\in \Lambda}\log(qm+r) \ e^{-2\pi i h_q(m)t/d_k},$$ for which if suffices to show \begin{equation}\label{suff2} \sum_{t\neq 0}\widehat{f_1}(t)\cdots\widehat{f_{k-1}}(t)|V(t)| < \widehat{f_1}(0)\cdots\widehat{f_{k-1}}(0)V(0)=(\psi(qn+r,r,q) + O_h(\log N))\prod_{i=1}^{k-1}|A_i|. 
\end{equation} 
where the estimate on $V(0)$ follows from (\ref{symdif}), and we recall that $$\psi(x,a,q):=\sum_{\substack{p\in \P\cap[1,x] \\ p \equiv a \textnormal{ mod }q}} \log p.$$ We now apply Lemma \ref{min} with $U=C\delta^{-1}(4k)^k \log N$ for a sufficiently large constant $C$ depending on $h$, noting that if $A$ satisfies (\ref{minL2}) then $d_k > h_q(n)/U^L$, where $L=64\ell^24^{\ell}$. Therefore, by (\ref{finv}) and (\ref{alphaB}), if $(t,d_k)<d_k/U^L$ then 
\begin{equation}\label{ltX}\sum_{(t,d_k)\leq d_k/U^L} \widehat{f_1}(t)\cdots\widehat{f_{k-1}}(t)|V(t)|\ll_h (n/U+U^Ln^{1-4^{-\ell}})\prod_{i=1}^{k-1}|A_i|(4k)^k/\delta.
\end{equation}
Again we see that if $(t,d_k)>d_k/U^L$ and $d_k/(t,d_k) \nmid d_j$, then
$\widehat{f_j}(t) \leq U^{2L}/|A_j|$. Applying Lemma \ref{moment}, (\ref{finv}),  and H\"older's Inequality with $s=2^{\ell}+1$, we have
\begin{align*}\sum_{(t,d_k)\geq d_k/U^L} \widehat{f_1}(t)\cdots\widehat{f_{k-1}}(t)|V(t)| & \ll_h \Big( \sum_{t\in \Z_{d_k}}|V(t)|^s\Big)^{1/s}\Big(\sum_{(t,d_k)\geq d_k/U^L}(\widehat{f_1}(t)\cdots\widehat{f_{k-1}}(t))^{1+1/(s-1)} \Big)^{(s-1)/s} \\ & \ll_h nU^{2L/s}(\min\{|A_i|\}_{i=1}^{k-1})^{-2/s}\Big(\prod_{i=1}^{k-1}|A_i|\Big)^{1/s}(M^2d_k)^{(s-1)/s},
\end{align*}
which combined with (\ref{alphaB}) yields
\begin{equation} \label{gtX} \sum_{(t,d_k)\geq d_k/U^L} \widehat{f_1}(t)\cdots\widehat{f_{k-1}}(t)|V(t)|\ll_h nM^2U^{2L/s}(\min\{|A_i|\}_{i=1}^{k-1})^{-2/s}\prod_{i=1}^{k-1}|A_i|(4k)^k/\delta. 
\end{equation}

\noindent If $\delta \geq CN^{-c/\ell}$ for appropriately large and small constants, respectively, then in particular $n>N^{1/2\ell}$, and since $q\leq\delta^{-1}$ we know from Lemma \ref{qlin} that  
\begin{equation}\label{plb} \psi(qn+r,r,q) \gg \frac{qn}{\phi(q)\sqrt{q}} \geq \delta^{1/2}n. \end{equation}
We see that if $A$ satisfies (\ref{minL2}), then  in particular $\min\{|A_i|\}_{i=1}^{k-1} \geq U^{2L}(M(4k)^k/\delta)^s$. In this case, we see from (\ref{ltX}), (\ref{gtX}), and (\ref{plb})  that the left-hand side of (\ref{suff2}) is bounded by a small constant times the right-hand side of (\ref{suff2}), and the lemma follows. 
\end{proof}

\section{Some Remarks on Special Cases and Expected Bounds} \label{special}
Here we note that the ``wideness" conditions stipulated in Lemmas \ref{itn} and \ref{itn2} are stronger, simplified formulations of what are actually used in the respective proofs. Specifically, in the proof of Lemma \ref{itn}, we exhibited the following. 

\begin{lemma}\label{spec1}  Suppose $A=\{x_1D_1 + \cdots + x_kD_k : \ |x_i|\leq L_i \} \subseteq [-N,N]$ is an $M$-proper, symmetric GAP with $L_1\cdots L_k=\delta N$ and $L_kD_k \geq L_iD_i$ for $1\leq i \leq k$, suppose $h\in \Z[x]$ is a polynomial of degree $\ell$ which has a root modulo $q=(D_1,\dots,D_k)$, and let $d_i=D_i/q$. If $L_k \geq C\Big(M(4k)^{k}\delta^{-1}\log N\Big)^{2^{\ell+2}}$ and for every $D\geq d_k\Big(CM^{2^{\ell+1}}(4k)^{2^{\ell}k}(\log N)^{\ell^2}\delta^{-(2^{\ell}+\ell)}\Big)^{-1}$ with $D\mid d_k$, there exists $1\leq i \leq k-1$ with $D \nmid d_i$ and
\begin{equation}\label{fat3} L_i \geq C\Big(M(4k)^{k}\delta^{-1}\log N\Big)^{2^{\ell+2}} \end{equation} 
for a sufficiently large constant $C$ depending only on $h$, then $A$ contains a nonzero element of $h(\Z)$. 
\end{lemma}

\noindent One can also extract a sharper version of Lemma \ref{itn2}, and these more careful formulations allow one to deduce much improved bounds in  special cases in which the iteration procedure used to prove Theorems \ref{M1} and \ref{M2} is not necessary. In particular, if the GAP is equally ``wide" in each dimension, i.e. $$A=\{x_1d_1+\cdots+x_kd_k \ : \ |x_i|\leq L\} \subseteq [-N,N],$$  then one obtains bounds of the form \begin{equation}\label{imp} |A|\ll_h N^{1-c^{\ell}/k}, \end{equation} which hold for dimensions up to about $k=\sqrt{\log N}$. Of course, if the elements of the GAP are extremely concentrated in one dimension, then one immediately has excellent bounds by reducing to the one-dimensional case. The difficulty  lies in intermediate cases such as $$A=\{x_1d_1+\cdots+x_kd_k \ : \ |x_i|\leq L_i\}, \quad |A|=\delta N, \quad L_i \approx \delta^{-2^i}, $$ but we believe this to be a shortcoming of the proof rather than a genuine obstruction, and it is likely that, in both the unrestricted  and prime input settings, bounds of the form (\ref{imp}) hold in full generality.


\begin{thebibliography}{10}  
\bibitem{BPPS} 
{\sc A. Balog, J. Pelik\'an, J. Pintz, E. Szemer\'edi}, {\em Difference Sets Without $k$-th Powers}, Acta. Math. Hungar. 65 (2) (1994), pp. 165-187.

\bibitem{CLS}
{\sc E. Croot, I. \L aba, O. Sisask}, {\em Arithmetic progressions in sumsets and $L^p$ almost periodicity}, Combinatorics, Probability, and Computing 22 (2013), 351-365.


\bibitem{GT} 
{\sc B. Green, T. Tao}, {\em New bounds for Szemer\'edi's theorem II. A new bound for $r_4(N)$,} Analytic number theory, 180-204, Cambridge Univ. Press, 2009.

\bibitem{HLR}
{\sc M. Hamel, N. Lyall, A. Rice}, {\em Improved bounds on S\'ark\"ozy's theorem for quadratic polynomials}, Int. Math. Res. Not. no. 8 (2013), 1761-1782 

\bibitem{Hua}
{\sc L. K. Hua}, {\em Additive theory of prime numbers }, American Mathematical Society, Providence, RI 1965.

\bibitem{IK}
{\sc H. Iwaniec, E. Kowalski}, {\em Analytic number theory}, AMS Colloquium Publications Volume 53, American Mathematical Society, Providence, Rhode Island, 2004.

\bibitem{Le}
{\sc T. H. L\^{e}}, {\em Intersective polynomials and the primes}, J. Number Theory 130 no. 8 (2010), pp. 1705-1717.  

\bibitem{lipan} 
{\sc H.-Z. Li, H. Pan}, {\em Difference sets and polynomials of prime variables}, Acta. Arith. 138, no. 1 (2009), 25-52. 

\bibitem{Lucier2}
{\sc J. Lucier}, {\em Difference sets and shifted primes}, Acta. Math. Hungar. 120 (2008), 79-102.

\bibitem{Lucier}
{\sc J. Lucier}, {\em Intersective Sets Given by a Polynomial}, Acta Arith. 123 (2006), pp. 57-95.

\bibitem{Lyall}
{\sc N. Lyall}, {\em A new proof of S\'ark\"ozy's theorem}, Proc. Amer. Math. Soc. 141 (2013), 2253-2264.

\bibitem{LM2}
{\sc N. Lyall, \`A. Magyar}, {\em Polynomial configurations in difference sets}, J. Number Theory 129 (2009), pp. 439-450. 

\bibitem{LR}
{\sc N. Lyall, A. Rice}, {\em A quantitative result on diophantine approximation for intersective polynomials}, preprint available at \texttt{arxiv.org/abs/1404.5161}



\bibitem{PSS}
{\sc J. Pintz, W. L. Steiger, E. Szemer\'edi}, {\em On sets of natural numbers whose difference set contains no squares}, J. London Math. Soc. 37 (1988), pp. 219-231.

\bibitem{thesis}
{\sc A. Rice}, {\em Improvements and extensions of two theorems of S\'ark\"ozy}, Ph.D. thesis, University of Georgia, 2012, available at \texttt{alexricemath.com/wp-content/uploads/2013/06/AlexThesis.pdf} 

\bibitem{Rice} 
{\sc A. Rice}, {\em S\'ark\"ozy's theorem for $\P$-intersective polynomials}, Acta Arith. 157 (2013), no. 1, 69-89. 

\bibitem{Ruz}
{\sc I. Ruzsa, T. Sanders}, {\em Difference sets and the primes}, Acta. Arith. 131, no. 3 (2008), 281-201.

\bibitem{Sanders}
{\sc T. Sanders}, {\em On the Bogolyubov-Ruzsa lemma}, Anal. PDE 5 (2012), no. 3, 627-655.

\bibitem{Sark1}
{\sc A. S\'ark\"ozy}, {\em On difference sets of sequences of integers I}, Acta. Math. Hungar. 31(1-2) (1978), pp. 125-149.

\bibitem{Sark3} 
{\sc A. S\'ark\"ozy}, {\em On difference sets of sequences of integers III}, Acta. Math. Hungar. 31(3-4) (1978), pp. 355-386.

\bibitem{Schmidt}
{\sc W. M. Schmidt}, {\em Small fractional parts of polynomials},  CBMS Regional Conference Series in Math., {\bf 32}, Amer. Math. Soc., 1977.  

\bibitem{Slip} {\sc S. Slijep\v{c}evi\'c}, {\em A polynomial S\'ark\"ozy-Furstenberg theorem with upper bounds}, Acta Math. Hungar. 98 (2003), pp. 275-280

\bibitem{TV}
{\sc T. Tao, V. Vu}, {\em Additive combinatorics}, Cambridge Studies in Advanced Mathematics 105, Cambridge University Press, paperback, 2010.

\bibitem{vaughan}
{\sc R. C. Vaughan}, {\em The Hardy-Littlewood method}, Cambridge University Press, second edition, 1997.

\bibitem{xyl}

 {\sc T. Xylouris}, {\em The zeros of Dirichlet L-functions and the least prime in an arithmetic progression}, Ph. D. thesis, Universität Bonn, Mathematisches Institut, 2011.

\end{thebibliography}
\end{document}